\documentclass[12pt,a4paper]{article}
\usepackage{ifthen,latexsym,amssymb,amsmath}
\usepackage{amsfonts}
\usepackage{amsthm}
\usepackage{bbm}
\usepackage{url}
\usepackage[affil-it]{authblk}
\usepackage{enumitem}
\usepackage{bm,bigints}
\usepackage[nottoc]{tocbibind}
\usepackage{graphicx,float}
\usepackage{MnSymbol}
\usepackage{etoolbox}
\usepackage{tikz}
\usepackage{float}
\usetikzlibrary{calc}
\AtBeginEnvironment{quote}{\itshape}

\DeclareMathAlphabet\EuScript{U}{eus}{m}{n}
\SetMathAlphabet\EuScript{bold}{U}{eus}{b}{n}

\DeclareSymbolFont{rsfs}{U}{rsfs}{m}{n}
\DeclareSymbolFontAlphabet{\mathrsfs}{rsfs}

\theoremstyle{plain}
\newtheorem{theorem}{Theorem}
\newtheorem{lemma}[theorem]{Lemma}

\newtheorem{remark}[theorem]{Remark}

\theoremstyle{remark}

\newcommand{\R}{\mathbb{R}}
\newcommand{\C}{\mathbb{C}}

\newcommand{\Hyp}{\mathbb{H}}
\newcommand{\e}{\varepsilon}

\newcommand{\hyp}{\mathbb{H}}

\newif\ifnotesw\noteswtrue

\noteswfalse	
\newcommand{\hide}[1]{}

\begin{document}
\title{Lower bounds for the measurable chromatic number of the hyperbolic plane}
\author[1]{Evan DeCorte}
\affil[1]{McGill University}
\author[2]{Konstantin Golubev}
\affil[2]{Bar-Ilan University and Weizmann Institute of Science}
\date{\today}

\maketitle

\abstract{Consider the graph $\Hyp(d)$ whose vertex set is the hyperbolic plane, where two points are connected with an edge when their distance is equal to some $d>0$. Asking for the chromatic number of this graph is the hyperbolic analogue to the famous Hadwiger-Nelson problem about colouring the points of the Euclidean plane so that points at distance $1$ receive different colours.

As in the Euclidean case, one can lower bound the chromatic number of $\Hyp(d)$ by $4$ for all $d$. Using spectral methods, we prove that if the colour classes are measurable, then at least $6$ colours are are needed to properly colour $\Hyp(d)$ when $d$ is sufficiently large.}

\section{Introduction}

The Hadwiger-Nelson problem \cite{Soifer2016}, posed in the 1950's, asks for the smallest possible number of colours required to paint the points of $\R^2$ so that no two points at distance $1$ receive
the same colour. A lower bound of $4$ is obtained by noticing that any $3$-colouring
of $\R^2$ would induce a $3$-colouring of the the Moser spindle (see Figure~\ref{fig:MoserSpindle}).
On the other hand, one can obtain an upper bound of $7$ by colouring the hexagons in a hexagonal tiling of $\R^2$. Amazingly, no improvement to these bounds has been made since the
original statement of the problem.
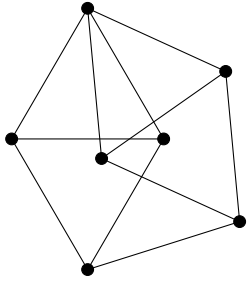
\begin{figure}[H]
\centering
\begin{tikzpicture}[scale=2,rotate=270]

	\draw (0:0)--(-30:1)--(30:1)--(0:0);
	\draw (-30:1)--(0:{sqrt(3)})--(30:1);
	\draw[rotate={asin(sqrt(3)/3)}] (0:0)--(-30:1)--(30:1)--(0:0);
	\draw[rotate={asin(sqrt(3)/3)}] (-30:1)--(0:{sqrt(3)})--(30:1););
	\draw (0:{sqrt(3)})--({asin(sqrt(3)/3)}:{sqrt(3)});
	
	\node[draw,fill,circle,scale=0.4, fill=black] at (0:0) (){};
	\node[draw,fill,circle,scale=0.4, fill=black] at (-30:1) (){};
	\node[draw,fill,circle,scale=0.4, fill=black] at (30:1) (){};
	\node[draw,fill,circle,scale=0.4, fill=black] at ({-30+asin(sqrt(3)/3)}:1) (){};
	\node[draw,fill,circle,scale=0.4, fill=black] at ({30+asin(sqrt(3)/3)}:1) (){};
	\node[draw,fill,circle,scale=0.4, fill=black] at (0:{sqrt(3)}) (){};
	\node[draw,fill,circle,scale=0.4, fill=black] at ({asin(sqrt(3)/3)}:{sqrt(3)}) (){};

\end{tikzpicture}
\caption{Moser spindle: a four-chromatic unit-distance graph.} \label{fig:MoserSpindle}
\end{figure}

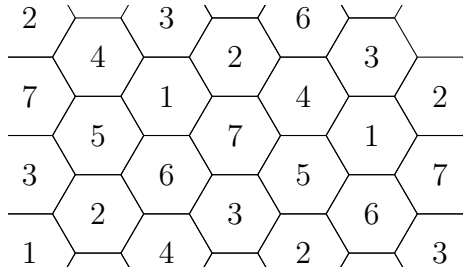
\begin{figure}[H]
\centering
\begin{tikzpicture}[scale=0.6,rotate=0]
	\clip (150:4) rectangle (-30:7.5);

	\draw (0:1)--(60:1)--(120:1)--(180:1)--(240:1)--(300:1)--(0:1);
	
	\draw[shift={(30:{sqrt(3)})}](0:1)--(60:1)--(120:1)--(180:1)--(240:1)--(300:1)--(0:1);
	\draw[shift={(90:{sqrt(3)})}] (0:1)--(60:1)--(120:1)--(180:1)--(240:1)--(300:1)--(0:1);
	\draw[shift={(150:{sqrt(3)})}](0:1)--(60:1)--(120:1)--(180:1)--(240:1)--(300:1)--(0:1);
	\draw[shift={(210:{sqrt(3)})}](0:1)--(60:1)--(120:1)--(180:1)--(240:1)--(300:1)--(0:1);
	\draw[shift={(270:{sqrt(3)})}] (0:1)--(60:1)--(120:1)--(180:1)--(240:1)--(300:1)--(0:1);
	\draw[shift={(330:{sqrt(3)})}](0:1)--(60:1)--(120:1)--(180:1)--(240:1)--(300:1)--(0:1);

	\node (a) at (0,0) {1};
	\node (a) at (30:{sqrt(3)}) {2};
	\node (a) at (90:{sqrt(3)}) {3};
	\node (a) at (150:{sqrt(3)}) {4};
	\node (a) at (210:{sqrt(3)}) {5};
	\node (a) at (270:{sqrt(3)})  {6};
	\node (a) at (330:{sqrt(3)}) {7};
	
\begin{scope}[shift={(60:3)}]
\begin{scope}[shift={(30:{sqrt(3)})}]
	\draw (0:1)--(60:1)--(120:1)--(180:1)--(240:1)--(300:1)--(0:1);
	
	\draw[shift={(30:{sqrt(3)})}](0:1)--(60:1)--(120:1)--(180:1)--(240:1)--(300:1)--(0:1);
	\draw[shift={(90:{sqrt(3)})}] (0:1)--(60:1)--(120:1)--(180:1)--(240:1)--(300:1)--(0:1);
	\draw[shift={(150:{sqrt(3)})}](0:1)--(60:1)--(120:1)--(180:1)--(240:1)--(300:1)--(0:1);
	\draw[shift={(210:{sqrt(3)})}](0:1)--(60:1)--(120:1)--(180:1)--(240:1)--(300:1)--(0:1);
	\draw[shift={(270:{sqrt(3)})}] (0:1)--(60:1)--(120:1)--(180:1)--(240:1)--(300:1)--(0:1);
	\draw[shift={(330:{sqrt(3)})}](0:1)--(60:1)--(120:1)--(180:1)--(240:1)--(300:1)--(0:1);

	\node (a) at (0,0) {1};
	\node (a) at (30:{sqrt(3)}) {2};
	\node (a) at (90:{sqrt(3)}) {3};
	\node (a) at (150:{sqrt(3)}) {4};
	\node (a) at (210:{sqrt(3)}) {5};
	\node (a) at (270:{sqrt(3)})  {6};
	\node (a) at (330:{sqrt(3)}) {7};
\end{scope}		
\end{scope}	

\begin{scope}[shift={(180:3)}]
\begin{scope}[shift={(150:{sqrt(3)})}]
	\draw (0:1)--(60:1)--(120:1)--(180:1)--(240:1)--(300:1)--(0:1);
	
	\draw[shift={(30:{sqrt(3)})}](0:1)--(60:1)--(120:1)--(180:1)--(240:1)--(300:1)--(0:1);
	\draw[shift={(90:{sqrt(3)})}] (0:1)--(60:1)--(120:1)--(180:1)--(240:1)--(300:1)--(0:1);
	\draw[shift={(150:{sqrt(3)})}](0:1)--(60:1)--(120:1)--(180:1)--(240:1)--(300:1)--(0:1);
	\draw[shift={(210:{sqrt(3)})}](0:1)--(60:1)--(120:1)--(180:1)--(240:1)--(300:1)--(0:1);
	\draw[shift={(270:{sqrt(3)})}] (0:1)--(60:1)--(120:1)--(180:1)--(240:1)--(300:1)--(0:1);
	\draw[shift={(330:{sqrt(3)})}](0:1)--(60:1)--(120:1)--(180:1)--(240:1)--(300:1)--(0:1);

	\node (a) at (0,0) {1};
	\node (a) at (30:{sqrt(3)}) {2};
	\node (a) at (90:{sqrt(3)}) {3};
	\node (a) at (150:{sqrt(3)}) {4};
	\node (a) at (210:{sqrt(3)}) {5};
	\node (a) at (270:{sqrt(3)})  {6};
	\node (a) at (330:{sqrt(3)}) {7};
\end{scope}		
\end{scope}

\begin{scope}[shift={(0:3)}]
\begin{scope}[shift={(-30:{sqrt(3)})}]
	\draw (0:1)--(60:1)--(120:1)--(180:1)--(240:1)--(300:1)--(0:1);
	
	\draw[shift={(30:{sqrt(3)})}](0:1)--(60:1)--(120:1)--(180:1)--(240:1)--(300:1)--(0:1);
	\draw[shift={(90:{sqrt(3)})}] (0:1)--(60:1)--(120:1)--(180:1)--(240:1)--(300:1)--(0:1);
	\draw[shift={(150:{sqrt(3)})}](0:1)--(60:1)--(120:1)--(180:1)--(240:1)--(300:1)--(0:1);
	\draw[shift={(210:{sqrt(3)})}](0:1)--(60:1)--(120:1)--(180:1)--(240:1)--(300:1)--(0:1);
	\draw[shift={(270:{sqrt(3)})}] (0:1)--(60:1)--(120:1)--(180:1)--(240:1)--(300:1)--(0:1);
	\draw[shift={(330:{sqrt(3)})}](0:1)--(60:1)--(120:1)--(180:1)--(240:1)--(300:1)--(0:1);

	\node (a) at (0,0) {1};
	\node (a) at (30:{sqrt(3)}) {2};
	\node (a) at (90:{sqrt(3)}) {3};
	\node (a) at (150:{sqrt(3)}) {4};
	\node (a) at (210:{sqrt(3)}) {5};
	\node (a) at (270:{sqrt(3)})  {6};
	\node (a) at (330:{sqrt(3)}) {7};
\end{scope}		
\end{scope}	

\begin{scope}[shift={(-60:3)}]
\begin{scope}[shift={(-90:{sqrt(3)})}]
	\draw (0:1)--(60:1)--(120:1)--(180:1)--(240:1)--(300:1)--(0:1);
	
	\draw[shift={(30:{sqrt(3)})}](0:1)--(60:1)--(120:1)--(180:1)--(240:1)--(300:1)--(0:1);
	\draw[shift={(90:{sqrt(3)})}] (0:1)--(60:1)--(120:1)--(180:1)--(240:1)--(300:1)--(0:1);
	\draw[shift={(150:{sqrt(3)})}](0:1)--(60:1)--(120:1)--(180:1)--(240:1)--(300:1)--(0:1);
	\draw[shift={(210:{sqrt(3)})}](0:1)--(60:1)--(120:1)--(180:1)--(240:1)--(300:1)--(0:1);
	\draw[shift={(270:{sqrt(3)})}] (0:1)--(60:1)--(120:1)--(180:1)--(240:1)--(300:1)--(0:1);
	\draw[shift={(330:{sqrt(3)})}](0:1)--(60:1)--(120:1)--(180:1)--(240:1)--(300:1)--(0:1);

	\node (a) at (0,0) {1};
	\node (a) at (30:{sqrt(3)}) {2};
	\node (a) at (90:{sqrt(3)}) {3};
	\node (a) at (150:{sqrt(3)}) {4};
	\node (a) at (210:{sqrt(3)}) {5};
	\node (a) at (270:{sqrt(3)})  {6};
	\node (a) at (330:{sqrt(3)}) {7};
\end{scope}		
\end{scope}	

\begin{scope}[shift={(-120:3)}]
\begin{scope}[shift={(-150:{sqrt(3)})}]
	\draw (0:1)--(60:1)--(120:1)--(180:1)--(240:1)--(300:1)--(0:1);
	
	\draw[shift={(30:{sqrt(3)})}](0:1)--(60:1)--(120:1)--(180:1)--(240:1)--(300:1)--(0:1);
	\draw[shift={(90:{sqrt(3)})}] (0:1)--(60:1)--(120:1)--(180:1)--(240:1)--(300:1)--(0:1);
	\draw[shift={(150:{sqrt(3)})}](0:1)--(60:1)--(120:1)--(180:1)--(240:1)--(300:1)--(0:1);
	\draw[shift={(210:{sqrt(3)})}](0:1)--(60:1)--(120:1)--(180:1)--(240:1)--(300:1)--(0:1);
	\draw[shift={(270:{sqrt(3)})}] (0:1)--(60:1)--(120:1)--(180:1)--(240:1)--(300:1)--(0:1);
	\draw[shift={(330:{sqrt(3)})}](0:1)--(60:1)--(120:1)--(180:1)--(240:1)--(300:1)--(0:1);

	\node (a) at (0,0) {1};
	\node (a) at (30:{sqrt(3)}) {2};
	\node (a) at (90:{sqrt(3)}) {3};
	\node (a) at (150:{sqrt(3)}) {4};
	\node (a) at (210:{sqrt(3)}) {5};
	\node (a) at (270:{sqrt(3)})  {6};
	\node (a) at (330:{sqrt(3)}) {7};
\end{scope}		
\end{scope}	

\begin{scope}[shift={(0:6)}]
\begin{scope}[shift={(-90:{3*sqrt(3)})}]
	\draw (0:1)--(60:1)--(120:1)--(180:1)--(240:1)--(300:1)--(0:1);
	
	\draw[shift={(30:{sqrt(3)})}](0:1)--(60:1)--(120:1)--(180:1)--(240:1)--(300:1)--(0:1);
	\draw[shift={(90:{sqrt(3)})}] (0:1)--(60:1)--(120:1)--(180:1)--(240:1)--(300:1)--(0:1);
	\draw[shift={(150:{sqrt(3)})}](0:1)--(60:1)--(120:1)--(180:1)--(240:1)--(300:1)--(0:1);
	\draw[shift={(210:{sqrt(3)})}](0:1)--(60:1)--(120:1)--(180:1)--(240:1)--(300:1)--(0:1);
	\draw[shift={(270:{sqrt(3)})}] (0:1)--(60:1)--(120:1)--(180:1)--(240:1)--(300:1)--(0:1);
	\draw[shift={(330:{sqrt(3)})}](0:1)--(60:1)--(120:1)--(180:1)--(240:1)--(300:1)--(0:1);

	\node (a) at (0,0) {1};
	\node (a) at (30:{sqrt(3)}) {2};
	\node (a) at (90:{sqrt(3)}) {3};
	\node (a) at (150:{sqrt(3)}) {4};
	\node (a) at (210:{sqrt(3)}) {5};
	\node (a) at (270:{sqrt(3)})  {6};
	\node (a) at (330:{sqrt(3)}) {7};
\end{scope}		
\end{scope}	

\end{tikzpicture}
\caption{A regular hexagonal tesselation with the diameter of the hexagons slightly less than 1 provides a 7-colouring of the Euclidean plane. } \label{fig:HexTiling}
\end{figure}

The analogous problem for the hyperbolic plane was asked by
Matthew Kahle on Math Overflow \cite{kahle12}, and listed as Problem P in~\cite{kloeckner15}.
Precisely, he asks,
\begin{quote}
	Let $\hyp$ be the hyperbolic plane (with constant curvature $-1$), and
	let $d>0$ be given. If $\hyp(d)$ is the graph on vertex set $\hyp$, in which two
	points are joined with an edge when they have distance $d$, then what is
	the chromatic number $\chi(\hyp(d))$?
\end{quote}

Note that unlike in the Euclidean case, the curvature of the hyperbolic plane
makes the problem different for each value of $d$.
This problem appears to be comparable in difficultly to the original Hadwiger-Nelson
problem. Indeed, the best known lower
bound is still $4$, again given by the Moser spindle. The most efficient colourings,
given in \cite{kloeckner15} and \cite{parlier17}, 
come from a hyperbolic checkerboard construction, and the number of colours
increases linearly with $d$: for large enough $d$, we have
\[
	4 \leq \chi(\hyp(d)) \leq 5 \left( \lceil \frac{d}{\log 4} \rceil + 1 \right).
\]

Our work focusses on the lower bound, about which nothing has
been published as far as we know. Our result is the following.

\begin{theorem}\label{thm:main}
Let $d > 0$ be given and suppose that $C_1, \dots, C_k$ are measurable
subsets of $\Hyp$ (with respect to the Haar measure)
such that $\Hyp = \cup_{i} C_i$ and such that no $C_i$
contains a pair of points at distance $d$. Then $k \geq 6$ provided that
$d$ is sufficiently large.
\end{theorem}

When working with infinite graphs having a measurable structure,
the smallest possible value of $k$ in Theorem~\ref{thm:main} is called the
\emph{measurable chromatic number}, in order to distinguish it from the honest
chromatic number. Falconer~\cite{falconer1981} in 1981 proved that if the colour
classes in a proper
colouring of $\R^2$ are assumed to be Lebesgue measurable, then at least five colours
are required. Theorem~\ref{thm:main} can therefore be regarded as a
Falconer-type result about the measurable chromatic number of the hyperbolic plane.

While Falconer's argument can be carried out in the hyperbolic setting, it unfortunately only
yields $k \geq 5$. Our result is an application of the spectral method for sparse infinite
graphs, a relatively new approach which has seen several successes in the past decade.
In the series of articles \cite{bachoc09}, \cite{oliveira09}, \cite{filho+vallentin:10},
the authors compute upper bounds on the
maximum density of a distance-$1$ avoiding subset of Euclidean space $\R^n$
for several values of $n$,
which in turn produce lower bounds on the measurable chromatic number.
The current best known upper bound in the $n=2$ case, again obtained via
the spectral method, is given in~\cite{keleti2016}.
In~\cite{bachoc2015}, the authors apply the spectral method to
produce the first asymptotic improvement to
the upper bounds for distance-$1$ avoiding sets
since~\cite{frankl-wilson81} and~\cite{raigorodskii2000chromatic}.

One may ask the analogous question for the sphere. Namely, what is the largest
possible surface measure of a subset of $S^{n-1} \subset \R^n$ not containing
any pair of points lying at some prescribed angle $\theta$? For $\theta = \pi/2$,
this is known as the Witsenhausen problem, first stated in~\cite{witsenhausen74}.
A general spectral approach
for computing upper bounds in the spherical case for any $n$ and $\theta$
is given in~\cite{bachoc09}, and this was refined in~\cite{decorte2015}
and~\cite{decorte2015thesis}
to give the first progress in over 40 years on the $n=3$ case of the Witsenhausen problem.

In~\cite{oliveira+vallentin13} the authors use the spectral method to prove
a quantitative version of Steinhaus's theorem for a large family of Riemannian manifolds,
which roughly speaking says that if one adds distances tending to zero
one at a time to a forbidden distance set, then the maximum density of a set avoiding
all the distances also tends to zero.

The spectral method can basically be regarded as an extension
to infinite graphs of Hoffman's eigenvalue bound~\cite{hoffman1970eigenvalues}
(or, in more sophisticated examples,
the Lov\'asz~$\vartheta$-function~\cite{lovasz79} bound and refinements thereof)
for the independence number or the chromatic number of a graph.

The proof of Theorem~\ref{thm:main} essentially consists of estimating the Hoffman
bound for the graphs $\Hyp(d)$.
This estimation is however not straightforward, and is in fact the main content of
this article.
This appears to be the first application of the spectral method in the sparse, non-compact,
non-abelian setting. It is also one of the only computer-free applications, and one
of the only applications in which a family of bounds is analyzed asymptotically.
A brief review of the spectral method in general will be given in
Section~\ref{sec:hoffmangeneral}; for a detailed treatment, we refer the reader to~\cite{bachoc14}.

This article is organized as follows: In Section \ref{sec:prelim}, we provide
the reader with some essential background in hyperbolic geometry, and in harmonic
analysis on the hyperbolic plane. We also briefly review the required parts
of the spectral method for sparse infinite graphs. The proof of Theorem~\ref{thm:main}
is given in Section~\ref{sec:proof} in a series of lemmas.
In Section~\ref{sec:openproblems} we present some open problems.


\section{Preliminaries}\label{sec:prelim}

\subsection{Hyperbolic Plane}


%

While several models of the hyperbolic plane are available, we shall make use only
of the Poincar\'e disk model because our computations can be carried out most easily there. A detailed introduction to the geometry of hyperbolic plane can be found in, for example,~\cite{helgason1984groups,katok1992fuchsian}.

Let $D = \{z=x+iy\in\C \mid |z| < 1\}$ be the open unit disc on the complex plane and let
$O$ denote the origin. As usual, $\overline{z}$ stands for the complex conjugate of $z$, and $|z|$ for the absolute value of $z$. Endowing $D$ with the Riemannian measure
\[
	dz = 4(1-x^2-y^2)^{-2}~dx~dy
\]
makes it a model of the hyperbolic plane $\Hyp$ with constant sectional curvature $-1$. (In~\cite{helgason1984groups}, the author omits the constant $4$, thus working in curvature $-4$. This implies rather minor differences in some formulae.)
The geodesics in this metric are the straight lines and arcs orthogonal to the boundary of $D$, the unit circle $B = \{z\in\C \mid |z|=1\}$. 

The hyperbolic distance from $O$ to a point $z \in D$ is equal to $\ln\left(\frac{1+|z|}{1-|z|}\right)$. The hyperbolic circle of radius $d$ with centre $O$ is the Euclidean circle of radius $\tanh(\frac{d}{2}) = \frac{e^d-1}{e^d+1}$ centred at $O$. The circumference of a hyperbolic
circle of radius $d$ is $2\pi \sinh(d)$, while the area is equal to $2\pi (\cosh(d)-1)$. Note that both the circumference and the area of a hyperbolic circle grows exponentially with radius.

Unlike the Euclidean plane, the sum of the angles of a hyperbolic triangles is always smaller than $\pi$. Moreover, the area of a triangle is equal to its angle deficit, i.e., $\pi$ minus the sum of the angles. A triangle is fully defined by its angles up to translations and reflections. Hence, in particular, it is not scalable. A triangle can also have a vertex at infinity, in this case the angle at this vertex is equal to 0. A triangle having all its vertices at infinity is called ideal. 

For a point on the plane $z\in D$ and a point on the boundary $b\in B$, let us denote by $H(z,b)$ a Euclidean circle in $D$ passing through $z$ and tangent to the boundary at $b$. It is called a horocycle. For any $z \in D$ a horocycle $H(z,b)$ is orthogonal to any geodesic emerging from $b$. Two horocycles $H(z_1,b)$ and $H(z_2,b)$ cut segments of equal length on geodesics emerging from $b$. In a way, horocycles are similar to hyperplanes orthogonal to a family of parallel lines in a Euclidean space. For $z\in D$ and $b\in B$, we denote by $\langle z,b\rangle$ the hyperbolic distance from the origin $O$ to the horocycle $H(z,b)$, if $O$ lies outside $H(z,b)$, and minus the distance, if $O$ lies inside it. The bracket $\langle z,b \rangle$ plays a crucial role in the definition of the Helgason-Fourier transform given below.

Let $G$ be the group defined by
$$
G = SU(1,1) = \left\{ \left( \begin{array}{ccc}
a & b \\
\overline{b} & \overline{a}  \end{array} \right) \in M_2(\mathbb{C})\mid |a|^2 - |b|^2 = 1 \right\},
$$
where $M_2(\C)$ is the set of $2 \times 2$ matrices over $\C$, with the usual matrix multiplication as the group operation.

We have the following group action $G \times D \to D$ of $G$ on $D$:
$$
g\cdot z = \left( \begin{array}{ccc}
a & b \\
\overline{b} & \overline{a}  \end{array} \right) \cdot z = \frac{az+b}{\overline{b}z+\overline{a}},
$$
for $g\in G$ and $z \in D$. This action is transitive and preserves the hyperbolic metric.
The subgroup of rotations
$$
K = \left\{ \left( \begin{array}{ccc}
a & 0\\
0 & \overline{a}  \end{array} \right) \in M_2(\mathbb{C})\mid |a|^2  = 1 \right\} \simeq SO(2)\subseteq G.
$$
is the stabilizer of the origin $O$. Hence $\Hyp$ can be identified with the coset space $G/K$.
We will use $dg$ and $dk$ to denote, respectively, the Haar measures on
$G$ and $K$, normalized so that $\int_K dk=1$. Then
$$
\int_G f(g\cdot O) dg = \int_\Hyp f(z)dz,
$$
whenever the integral exists.\hide{\footnote{p. 38 Helgason}}

\subsection{Helgason-Fourier Transform}

Mainly following~\cite[Section 0.4]{helgason1984groups}, we give a brief summary of harmonic
analysis on the hyperbolic plane. Note that the metric in~\cite{helgason1984groups} differs from the one above by the factor of 4, and hence some of the formulae differ from those in
that book.


As before we use the Riemannian measure
\[
	dz = 4(1-x^2-y^2)^{-2}~dx~dy
\]
on $D$. If $f$ is a complex-valued function on $D$,
then its \emph{Helgason-Fourier transform} $\widehat{f}(\lambda,b)$ is
the function on $\mathbb{C}\times B$ defined by
$$
\widehat{f}(\lambda,b) = \int_D f(z) e^{(-i\lambda+1/2)\left<z,b\right>} dz,
$$
for all $(\lambda,b)\in \mathbb{C}\times B$, for which the integral exists.
If $\mu$ is a finite Borel measure on $G$, we similarly define
\begin{align*}
	\widehat{\mu}(\lambda, b) = \int_G e^{(-i\lambda+1/2)\left<g\cdot O,b\right>}~d\mu(g)
\end{align*}
Notice that $\widehat{\mu}$ is always a continuous function.

\hide{
We say the function $f$ is \emph{radial} if $f(z)$ depends only on the distance from $z$ to $O$.
If $f$ is radial, then its Helgason-Fourier transform is constant with respect to $b$,
and so we rewrite it as
$$
\widehat{f}(\lambda) = \int_\Hyp f(z) \varphi_{-\lambda}(z)dz,
$$
where $\varphi_\lambda(z) := \int_B e^{(i\lambda+1/2)\langle z,b\rangle}db$.
The functions $\varphi_\lambda$ are called \emph{spherical functions}.
}

We say the finite Borel measure $\mu$ on $D$ is \emph{radial} if $\mu(k\cdot X) = \mu(X)$
for all $k \in K$ and all Borel subsets $X$ of $D$.
If $\mu$ is radial, then its Helgason-Fourier transform is constant with respect to $b$,
and so we rewrite it as
$$
\widehat{\mu}(\lambda) = \int_G \varphi_{-\lambda}(g\cdot O)~d\mu(g),
$$
where $\varphi_\lambda(z) := \int_B e^{(i\lambda+1/2)\langle z,b\rangle}db$.
The functions $\varphi_\lambda$ are called \emph{spherical functions}.

If $f_1,f_2$ are two functions on $D$, then their \emph{convolution product}
is defined (\cite[p.~43]{helgason1984groups}) as
$$
(f_1\ast f_2)(g \cdot O) = \int_G f_1(h \cdot O) f_2(h^{-1} g \cdot O)~dh
$$
when the integral exists.

Convolution of functions with measures can also be defined:
If $\mu$ is a finite Borel measure on $G$, then the convolution of $f$ by $\mu$
on the right is defined as
\begin{align*} 
	(f * \mu)(g \cdot O) = \int_G f(h^{-1} g \cdot O)~d\mu(h)
\end{align*}
for each $g \in G$ for which the integral is defined.
One can prove (using, for instance, Minkowski's inequality for integrals)
that $f*\mu \in L^2(D)$ whenever $f \in L^2(D)$. 



The following theorem gives us an important property of the Helgason-Fourier transform
which we will need in the following sections.


\begin{theorem}[{\cite[Theorem 4.2(iii)]{helgason1984groups}}]\label{thm:plancherel}
The map $f\mapsto \hat{f}$ extends to an isometry from $L^2(D)$ onto 
$$
L^2\left(\mathbb{R}\times B, \frac{1}{4 \pi} \lambda \tanh\left(\frac{\pi}{2}\lambda\right) d\lambda db \right).
$$
\end{theorem}

\hide{
\begin{theorem}[{\cite[p. 48]{helgason1984groups}}]
Let $f, \phi \in L^2(D)$ and suppose that $\phi$ is radial. Then
\[
	\widehat{f * \phi}(\lambda, b) = \widehat{f}(\lambda, b)~\widehat{\phi}(\lambda).
\]
\end{theorem}
}

The next theorem partially recovers a familiar relationship between the
usual Fourier transform and convolution product which holds on abelian groups.

\begin{theorem}\label{thm:conv-mult}
Suppose $f \in L^2(D)$ and that $\mu$ is a radial finite Borel measure on
$G$. Then we have
\[
	\widehat{f * \mu}(\lambda, b) = \widehat{f}(\lambda, b)~\widehat{\mu}(\lambda).
\]
\end{theorem}
\begin{proof}
(Following {\cite[p. 48]{helgason1984groups}}.)

For any $\lambda \in \C$ and $b \in B$, we have
\begin{align*}
	\widehat{f * \mu}(\lambda, b) &= \int_G \int_G f(h^{-1} g \cdot O)~d\mu(h)
		e^{(-i\lambda + 1/2)\langle g\cdot O,b \rangle}~dg\\
		&= \int_G \int_G f(g \cdot O) e^{(-i\lambda + 1/2) \langle hg\cdot O, b \rangle}~dg~d\mu(h)
\end{align*}
Applying the identity \cite[Eq.~34]{helgason1984groups}
\[
	\langle hg\cdot O,b \rangle = \langle g\cdot O, h^{-1}\cdot b \rangle
	+ \langle h\cdot O, b \rangle,
\]
we obtain
\begin{align*}
	\widehat{f * \mu}(\lambda, b) &=
	\int_G \int_G f(g\cdot O)
	e^{(-i\lambda + 1/2) \langle g\cdot O, h^{-1}\cdot b \rangle
	+ \langle h\cdot O, b \rangle}~dg~d\mu(h)\\
	&= \int_G \left( \int_G  f(g\cdot O)
	e^{(-i\lambda + 1/2) \langle g\cdot O, h^{-1}\cdot b \rangle}~dg \right)
	e^{(-i\lambda + 1/2) \langle h\cdot O, b \rangle}~d\mu(h\\
	&= \int_G \left( \int_G  f(g\cdot O)
	e^{(-i\lambda + 1/2) \langle g\cdot O, \cdot b \rangle}~dg \right)
	e^{(-i\lambda + 1/2) \langle h\cdot O, b \rangle}~d\mu(h))\tag{$*$}\label{eq:radial}\\
	&= \widehat{f}(\lambda, b)~\widehat{\mu}(\lambda),
\end{align*}
where we use the fact that $\mu$ is radial at line \eqref{eq:radial}.

\hide{
\begin{align*}
	\widehat{f * \mu}(\lambda, b) &= \int_G \int_G f(h g^{-1})~d\mu(g)
		~e^{-(i\lambda+1/2)\langle h\cdot O,b \rangle}~dh\\
	&= \int_G \int_G f(h g^{-1}) ~e^{(-i\lambda+1/2)\langle h\cdot O,b \rangle}~dh~d\mu(g)\\
	&= \int_G \int_G f(h) ~e^{(-i\lambda+1/2)\langle hg\cdot O,b \rangle}~dh~d\mu(g)\\
	&= \int_G \int_G f(h) ~e^{(-i\lambda+1/2) \langle g\cdot O, h^{-1}\cdot b \rangle
	+ \langle h\cdot O, b \rangle }~dh~d\mu(g)\\
	&= \int_G \left( \int_G f(h) e^{(-i\lambda+1/2) \langle h\cdot O, b \rangle}\right)
		e^{(-i\lambda+1/2) \langle g\cdot O, h^{-1}\cdot b \rangle}~dh~d\mu(g)\\
	&= \int_G \left( \int_G f(h) e^{(-i\lambda+1/2) \langle h\cdot O, b \rangle}~dh \right)
		e^{(-i\lambda+1/2) \langle g\cdot O, b \rangle}~d\mu(g)\tag{$*$}\label{eq:radial}\\
	&= \widehat{f}(\lambda, b) \int_G
		e^{(-i\lambda+1/2) \langle g\cdot O, b \rangle}~d\mu(g)\\
	&= \widehat{f}(\lambda,b)~\widehat{\mu}(\lambda),
\end{align*}
where we have used the identity \cite[Eq.~34]{helgason1984groups}
\[
	\langle hg\cdot O,b \rangle = \langle g\cdot O, h^{-1}\cdot b \rangle
	+ \langle h\cdot O, b \rangle.
\]
The fact the $\mu$ is radial is used at equation \eqref{eq:radial}. One can rotate
both $g$ and $h^{-1} \cdot b$, until the latter becomes $b$.
}

\end{proof}

\hide{
\begin{theorem}
Let $f,\phi \in L^2(D)$ and $\mu$ be a measure on $\Hyp$, where in addition $\phi$ and $\mu$ are radial, i.e., invariant under the action of $K$.  Then the following equalities holds
$$
\widehat{f\ast\phi}(\lambda,b) = \widehat{f}(\lambda,b)\cdot\widehat{\phi}(\lambda)\text{ and } \widehat{f * \mu}(\lambda, b) =  \widehat{f}(\lambda,b)~\widehat{\mu}(\lambda)
$$
where $\widehat{\mu}(\lambda) = \int_{G} e^{(-i\lambda+\frac{1}{2})\left<g\cdot O,b\right>} d\mu(g)$ (since $\mu$ is radial, $b$ can be chosen arbitrarily on $B$).
\end{theorem}
}

\subsection{The Hoffman bound for the measurable chromatic number of an infinite graph}\label{sec:hoffmangeneral}

The Hoffman bound is a well known spectral bound for the chromatic
number of a finite graph.
In this section we provide the reader the necessary background for understanding
the Hoffman bound applied to infinite graphs detailed in \cite{bachoc14}.
We begin with the usual statement of Hoffman's theorem for the finite context.

\begin{theorem}[\cite{hoffman1970eigenvalues}]\label{thm:hoffman-finite} 
	Let $G$ be a graph with adjacency matrix $A$. Denote the
	maximum (resp.~minimum) eigenvalue of $A$ by $M$
	(resp.~$m$), and denote the chromatic number of $G$
	by $\chi(G)$. Then
	\[
		\chi(G) \geq 1 - \frac{M}{m}.
	\]
\end{theorem}

The proof of Theorem \ref{thm:hoffman-finite} can be generalized to infinite
graphs, provided correct analogues of notions such as \emph{adjacency matrix}
and \emph{colouring} can be found. We now give those definitions and state
the infinite version of Hoffman's bound.

Now let $(X, \Sigma, \mu)$ be a measure space.
Write $L^2(X)$ for $L^2(X, \mu)$, we let\\$A~:~L^2(X)\to~L^2(X)$
be a bounded, self-adjoint operator. We say that a measurable subset
$I \subseteq X$ is \emph{$A$-independent} if
\begin{align}\label{eq:ind}
	\langle Af, g \rangle = 0
\end{align}
whenever $f,g \in L^2(X)$ are functions which vanish almost everywhere outside $I$.
When $I$ is an independent vertex subset of a finite graph $G$, and $A$
is the adjacency matrix of $G$, then one easily verifies \eqref{eq:ind} for all $f$ and
$g$ supported on $I$. One may therefore regard the operator $A$ as
playing the role of an adjacency operator.

An \emph{$A$-measurable colouring} is a partition $X = \bigcupdot_i C_i$
of $X$ into $A$-independent sets. In this case the sets $C_i$ are called
\emph{colour classes}. The \emph{$A$-chromatic number of $X$},
denoted $\chi_A(X)$, is the smallest number $k$ (possibly $\infty$), such that
there exists an $A$-measurable colouring using only $k$ colour classes.

Recall that operators on infinite-dimensional Hilbert spaces need not have eigenvectors,
even when they are self-adjoint, so Theorem~\ref{thm:hoffman-finite} does
not immediately extend to the infinite case. It turns out that for our purposes, the correct analogues
of the $M$ and $m$ of Theorem~\ref{thm:hoffman-finite}
are given by the following definitions.
\begin{align*}
	M(A) = \sup_{\|f\|_2 = 1} \langle Af, f \rangle,\\
	m(A) = \inf_{\|f\|_2=1} \langle Af, f \rangle.
\end{align*}

The numbers $\langle Af, f \rangle$ are real since $A$ is self-adjoint,
and $M(A)$ and $m(A)$ are finite since $A$ is bounded.
In \cite{bachoc14}, the following extension of Hoffman's bound is proven.

\begin{theorem}[\cite{bachoc14}]\label{thm:hoffman}
	Let $(X, \Sigma, \mu)$ be a measure space and suppose\\
	$A~:~L^2(X)\to~L^2(X)$ is a nonzero, bounded, self-adjoint operator.
	Then
	\[
		\chi_A(X) \geq 1 - \frac{M(A)}{m(A)}.
	\]
\end{theorem}

It should be noted that unlike in Theorem~\ref{thm:hoffman-finite} where
$A$ is simply the adjacency matrix, with infinite graphs there may be no
canonical choice for $A$.
The choice of $A$ determines which sets are to be considered as ``independent'',
and thus admissible as colour classes. When applying Theorem~\ref{thm:hoffman}
for an infinite graph $G$ on a measurable vertex set,
one therefore typically chooses $A$ to define a class of independent sets which is larger that the class of true measurable independent sets of $G$, for then the measurable chromatic
number of $G$ is at least $\chi_A(X)$. In most known interesting examples,
$A$-independent sets need not be honest independent sets in $G$.

Also note that Theorem~\ref{thm:hoffman} applies equally well
when $X$ is finite. Theorem~\ref{thm:hoffman} then says that
to obtain a lower bound on the chromatic number of a finite graph, one may
optimize over symmetric matrices satisfying \eqref{eq:ind}.

\section{Proof of the Main Theorem}\label{sec:proof}

\subsection{Adjacency on the hyperbolic plane}\label{sub:adj}

As before, for $d>0$ let $\Hyp(d)$ be the graph whose vertex set is the
hyperbolic plane $\Hyp$, where two points are joined with an edge precisely
when their distance is equal to $d$.
For any $f \in L^2(D)$, $d>0$, and $x \in \Hyp$, we define $(A_d f )(x)$ to be the
average of $f$ around the hyperbolic circle of radius $d$ centred at $x$. It is not hard to check
that $A$ is self-adjoint, and bounded with norm $1$. The operator $A_d$ can be thought
of as an adjacency operator for the graph $\Hyp(d)$, in the sense that any
measurable independent set of $\Hyp(d)$ is also $A_d$-independent.

If $C$ denotes the circle of radius $d$ centred at $O$, then
one can write $A_d f$ as $f * \mu$, where $\mu$ is the uniform probability
measure on the inverse image $p^{-1}(C)$ of $C$ under the canonical projection
map $p : G \to \Hyp$, given by $p(g) = g \cdot O$. One can also construct
$\mu$ as the Haar measure on the double coset
\[
	K\left( \begin{array}{ccc}
	\cosh(d/2) & \sinh(d/2)\\
	\sinh(d/2) & \cosh(d/2)  \end{array} \right)K.
\]


%

This allows us to apply Theorems~\ref{thm:conv-mult} and \ref{thm:plancherel}
for computing $M(A_d)$ and $m(A_d)$. Namely,
\begin{align*}
M(A_d) & = \sup_{\|f\|_2 = 1} \langle A_df, f \rangle = \sup_{\|f\|_2 = 1} \langle f*\mu_d, f \rangle\\
 & = \sup_{\|\widehat{f}\|_2 = 1} \langle \widehat{f}\cdot\widehat{\mu_d}, \widehat{f} \rangle = 
 \sup_{\lambda\in\R}\widehat{\mu_d}(\lambda),
\end{align*}
By definition,
$$
 \sup_{\lambda\in\R}\widehat{\mu_d}(\lambda) =  \sup_{\lambda\in\R}\int_{G} e^{(-i\lambda+\frac{1}{2})\left<g\cdot O,b\right>} d\mu_d(g) = \sup_{\lambda\in\R} \varphi_{-\lambda}(z_d),
$$
where $z_d\in\Hyp$ is any point at distance $d$ from the origin, and the $\varphi_\lambda$
are the spherical functions.
Similarly $m(A_d) = \inf_{\lambda\in\R} \varphi_{-\lambda}(z_d)$. 


\subsection{Hoffman calculation}

In this section we prove the Theorem~\ref{thm:main}, which is a direct corollary of the following one.

\begin{theorem}\label{thm:limit}
	Let $\nu =  \min_{s\in\R} \frac{\sin s}{s}$. Then
	\[
	\lim_{d \to \infty} \left( 1 - \frac{M(A_d)}{m(A_d)} \right)
	= 1 - 1/\nu \approx 5.6.
	\]
\end{theorem} 

\noindent In the subsection~\ref{sub:adj},
we showed that in order to find the Hoffman bound for the operator $A_d$, we need to calculate the global extrema of the function $\widehat{\mu_d}(\lambda) = \varphi_{-\lambda}(z_d)$. We start with the following lemma regarding spherical functions.

\begin{lemma}\label{lm:7} For $\lambda\in\R$ and $d>0$, the spherical function satisfies
$$
\varphi_{\lambda} (z_d) = \int_B e^{(1/2+i\lambda)\left< z,b\right>}db = \frac{\sqrt{2}}{\pi} \int_{0}^{1} f_d(v)\cos{(v\lambda d)}dv,
$$
where $z_d$ is any point at hyperbolic distance $d$ from the origin, and $$f_d(v) = \frac{d}{\sqrt{\cosh{(d)}-\cosh{(vd)}}}.$$
\end{lemma}
\begin{proof}

The equation~\cite[Introduction \S 4, Eqs.(19-20)]{helgason1984groups} reads as
$$
\varphi_{\lambda} (z) = \int_B e^{(1/2+i\lambda)\left< z,b\right>}db = \frac{1}{2\pi} \int_{-\pi}^{\pi} \left(\cosh{d} - \sinh{d}\cos{\theta}\right)^{-(1/2+i\lambda)} d\theta.
$$

Via the substitution $t = \cos{\theta}$ one gets 
$$
\frac{1}{\pi} \int_{0}^{\pi} \left(\cosh{d} - \sinh{d}\cos{\theta}\right)^{-(1/2+i\lambda)} d\theta = 
 \frac{1}{\pi} \int_{-1}^{1} \left(\cosh{d} - t\sinh{d}\right)^{-(1/2+i\lambda)} \frac{dt}{\sqrt{1-t^2}},
$$
and then by the substitution $u = \cosh{d} - t\sinh{d}$ and $v=\ln{u}$,
\begin{align*}
\varphi_{\lambda} (z) & = \frac{1}{\pi} \int_{\cosh{d}-\sinh{d}}^{\cosh{d}+\sinh{d}} \frac{u^{-(1/2+i\lambda)}}{\sqrt{2u\cosh{d}-1-u^2}}du = \frac{1}{\sqrt{2}\pi} \int_{e^{-d}}^{e^d} \frac{u^{-1-i\lambda}}{\sqrt{\cosh{d}- \frac{u+u^{-1}}{2}}} du \\
&= \frac{1}{\sqrt{2}\pi} \int_{-d}^{d} \frac{e^{-i v \lambda}}{\sqrt{\cosh{d}-\cosh{v}}}dv = \frac{\sqrt{2}}{\pi} \int_{0}^{1} d \frac{\cos{(vd\lambda)}}{\sqrt{\cosh{d}-\cosh{dv}}}dv,
\end{align*}
where the last equality holds, since the function $f_d(v)$ is even.
\end{proof}

Note that the function $f_d(v)$ is positive on $[0,1)$ for all $d$, and hence $\widehat{\mu_d}(\lambda)$ achieves its maximum at $\lambda=0$, and
$$
M(A_d) =  \frac{\sqrt{2}}{\pi} \int_{0}^{1} f_d(v)dv.
$$
Denote
$$
F_d(v) = \frac{f_d(v)}{\int_0^1 f_d(u)~du}.
$$

\noindent We proceed with a lemma.

\begin{lemma}\label{lm:l1}
The functions $F_d$ converge in $L^1$ norm to the constant function $1$ on $[0,1)$ as $d \to \infty$.
\end{lemma}
\begin{proof}
	First note that the functions $f_d$ are integrable. Indeed,
	examing the Taylor series of $\cosh$, one sees that
	\[
		f_d(v) = \frac{d}{\sqrt{\cosh(d) - \cosh(dv)}} \leq \frac{\sqrt{2}}{\sqrt{1-v^2}}
	\]
	for all $v \in [0,1)$.	
	We claim that
	\[
		F_d(v) = \frac{f_d(v)}{\int_0^1 f_d(u)~du} \leq \frac{1}{\sqrt{1-v}}
	\]
	for all $v \in [0,1)$, and all $d>0$. To see this, notice that the claim is
	equivalent to
	\begin{align}\label{eq:cosh}
		1-v \leq ( \cosh(d) - \cosh(dv)) \left( \int_0^1
		\frac{1}{\sqrt{\cosh(d) - \cosh(du)}}~du \right)^2.
	\end{align}
	Differentiating the righthand side of \eqref{eq:cosh} in $v$ twice gives
	\[
		-d^2 cosh(dv)  \left( \int_0^1
		\frac{1}{\sqrt{\cosh(d) - \cosh(du)}}~du \right)^2,
	\]
	which is always negative. So to prove the claim,
	it suffices to check $\eqref{eq:cosh}$ at $v=0$. But
	\[
		(\cosh(d) -1) \left( \int_0^1 \frac{1}{\sqrt{\cosh(d) - \cosh(du)}}~du \right)^2
		= \left( \int_0^1 \sqrt{\frac{\cosh(d)-1}{\cosh(d)-\cosh(du)}}~du \right)^2
		\geq 1,
	\]
	since the expression under the integral is always at least $1$.
	
	Now, applying dominated convergence twice, we get
	\begin{align*}
		\lim_{d \to \infty} \bigintssss_0^1 \left| \frac{f_d(v)}{\int_0^1 f_d(u)~du} - 1\right|~dv
		&= \lim_{d \to \infty} \bigintsss_0^1 \left|
		\left(\bigintssss_0^1 \sqrt{ \frac{\cosh(d) - \cosh(dv)}{\cosh(d) - \cosh(du)}}~du
		\right)^{-1} - 1\right|~dv \\
		&= \bigintsss_0^1 \lim_{d \to \infty} \left|
		\left(\bigintssss_0^1 \sqrt{ \frac{\cosh(d) - \cosh(dv)}{\cosh(d) - \cosh(du)}}~du
		\right)^{-1} - 1\right|~dv \\
		&= \bigintsss_0^1 \left|
		\left(\bigintssss_0^1 \lim_{d \to \infty}
		\sqrt{ \frac{\cosh(d) - \cosh(dv)}{\cosh(d) - \cosh(du)}}~du \right)^{-1} - 1\right|~dv\\
		&= \bigintsss_0^1 \left|
		\left(\bigintssss_0^1 \lim_{d \to \infty}
		\sqrt{ \frac{1 - \frac{\cosh(dv)}{\cosh(d)}}{1 - \frac{\cosh(du)}{\cosh(d)}}}~du
		 \right)^{-1} - 1\right|~dv,
	\end{align*}
	which is equal to $0$.
\end{proof}

\begin{proof}[Proof of Theorem \ref{thm:limit}]
For each $d>0$, let $s_d$ be the number which minimizes the expression
\[
	\int_0^1 f_d(v) \cos(s_d v)~dv
	= \bigintsss_0^1 \frac{d \cos(s_d v)}{\sqrt{\cosh(d) - \cosh(dv)}}~dv,
\]
and let $\rho$ be the unique positive minimizer for the expression
$\int_0^1 \cos(\rho v)~dv = \frac{\sin(\rho)}{\rho}$;
then $\nu =  \min_{s\in\R} \frac{\sin s}{s} = \int_0^1 \cos(\rho v)~dv \approx -0.217$.

It follows from Lemma \ref{lm:l1} that $s_d \to \rho$, by noting that
for any $\e > 0$ and for sufficiently large $d$, we have
\begin{align*}
	\left| \frac{\int_0^1 f_d(v) \cos(s v)~dv}{\int_0^1 f_d(v)~dv}
	- \int_0^1 \cos(s v)~dv \right|
	\leq \bigintssss_0^1 \left| \frac{f_d(v)}{\int_0^1 f_d(v)~dv} - 1 \right|~dv < \e
\end{align*}
for all $s$. Therefore $M(A_d)/m(A_d) \to 1/\nu$, and
the theorem follows.
\end{proof}

\hide{
\section{Numerical simulations}\label{sec:numerical}

Numerical simulations in Wolfram Mathematica show that the bound $1-\frac{M(A_d)}{m(A_d)}$ is monotonically increasing in $d$ and converges to the limit rather fast. Namely, for $d=4$ it is strictly larger than $4$, and for $d=12$ it is exceeds $5$. In other words, the measurable chromatic number is at least $6$ for as small value of $d$ as $12$. In the figure, you can see the graph of the bound for $1\leq d <200$.

\begin{figure}[H]
\centering 
	\includegraphics[scale=0.3]{explicit-plot.jpg}
\caption{The Hoffman bound for $A_d$ for $1\leq d <200$.} \label{fig:Explicilt-Plot}
\end{figure}
}

\begin{remark}\label{rmk:wolfram}
Numerical simulations in Wolfram Mathematica show that the bound $1-\frac{M(A_d)}{m(A_d)}$ is monotonically increasing in $d$ and converges quickly to the limit. In fact, for $d=4$ it is strictly larger than $4$, and for $d=12$ it is exceeds $5$. In other words, the measurable chromatic number is at least $6$ for values of $d$ as small as $12$. 
\end{remark}

\section{Open problems}\label{sec:openproblems}

We list here a few problems which are still open as far as we know.

\begin{enumerate}
\item Since the best known colourings of $\Hyp(d)$ use a number of colours growing linearly
in $d$, the following question is natural: Could our method be improved to give
$\chi_m(\Hyp(d)) \to \infty$  as $d\to\infty$?

\item Can a similar approach be used to find
lower bounds for $\chi_m(\Hyp^n(d))$, the measurable chromatic numbers of higher-dimensional
hyperbolic spaces?

\item Can the lower bound for the honest chromatic number, $\chi(\Hyp(d))$,
be improved from $4$, for any $d > 0$? This is actually Problem P from \cite{kloeckner15}.

\item As mentioned in Remark~\ref{rmk:wolfram}, we observed empirically that
the Hoffman bound increases monotonically in $d$, but we did not find a proof.
Can this be proven rigorously?

\item What is the limiting behaviour of the Hoffman bound as $d \to 0$? So far,
we have only considered the question of what happens when $d \to \infty$.

\end{enumerate}

\section*{Acknowledgements}
The authors would like to thank Shimon Brooks, Alexander Lubotzky, and Fernando Oliveira
for helpful comments and remarks, and Hillel Furstenberg for pointing out that a version
of Lemma~\ref{lm:7}
ought to be true. The work was started when the first author held the position of
Lady Davis Postdoctoral
Fellow at the Hebrew University of Jerusalem, and the second author was a Ph.D. student there.

\subsection*{Funding}
The first author is supported by the CRM Applied Math Laboratory and NSERC Discovery
Grant 2015-06746. During the first part of the work he was supported
by ERC grant GA 320924-ProGeoCom and the Lady Davis Fellowship Trust.
The second author is supported by the ERC grant 336283.

\bibliographystyle{alpha}
\bibliography{mybib}

\end{document}